\documentclass[10pt,reqno]{amsart}
\usepackage{hyperref}
\usepackage{graphicx}
\usepackage{color}
\usepackage[all]{xy}
\usepackage{amsfonts,amssymb}
\usepackage{amsthm}
\usepackage{bbm} 

\usepackage{mathrsfs}

\usepackage{pdfsync}

\usepackage{a4wide}

\newtheorem{neu}{}[section]

\newtheorem*{Cor*}{Corollary}
\newtheorem{Thm}[neu]{Theorem}
\newtheorem*{Thm*}{Theorem}

\newtheorem{Prop}[neu]{Proposition}
\newtheorem*{Prop*}{Proposition}
\theoremstyle{definition}
\newtheorem{Lemma}[neu]{Lemma}

\newtheorem*{Rmk*}{Remark}
\newtheorem{Rmk}[neu]{Remark}
\newtheorem{Ex}[neu]{Example}
\newtheorem*{Ex*}{Example}
\newtheorem{Qu}[neu]{Question}
\newtheorem*{Qu*}{Question}

\newtheorem{Def}[neu]{Definition}

\newcommand{\N}{\mathbb{N}}

\newcommand{\Z}{\mathbb{Z}}
\newcommand{\R}{\mathbb{R}}

\newcommand{\C}{\mathbb{C}}

\newcommand{\id}{\mathrm{id}}

\newcommand{\om}{\omega}
\newcommand{\Om}{\Omega}

\newcommand{\acs}{almost complex structure}

\renewcommand{\S}{\mathbb{S}}

\newcommand{\D}{\mathbb{D}}

\newcommand{\B}{\mathcal{B}}

\newcommand{\beq}{\begin{equation}}
\newcommand{\beqn}{\begin{equation}\nonumber}
\newcommand{\eeq}{\end{equation}}

\newcommand{\bea}{\begin{equation}\begin{aligned}}
\newcommand{\bean}{\begin{equation}\begin{aligned}\nonumber}
\newcommand{\eea}{\end{aligned}\end{equation}}

\numberwithin{equation}{section}

\definecolor{Urs}{rgb}{0,.7,0}
\definecolor{Youngjin}{rgb}{0,0,1}
\definecolor{red}{rgb}{1,0,0}

\newcommand{\p}{\partial}

\newcommand{\one}{{1\hskip-2.5pt{\rm l}}}

\begin{document}
\title[An overtwisted disk in a virtual contact structure and the Weinstein conjecture]
{An overtwisted disk in a virtual contact structure \\
and the Weinstein conjecture
}
\author{Youngjin Bae}
\address{
    Youngjin Bae\\
    Center for Geometry and Physic, Institute for Basic Science
	and Pohang University of Science and Technology(POSTECH)\\
	77 Cheongam-ro, Nam-gu, Pohang-si, Gyeongsangbuk-do, Korea 790-784}
\email{yjbae@ibs.re.kr}
\begin{abstract}
Hofer proved in \cite{Hof} the Weinstein conjecture for a closed contact 3-manifold with an overtwisted disk.
In this article we extend it to the virtual contact structure and provide a new explicit example of the virtual contact structure
with an overtwisted disk via a Lutz twist.
\end{abstract}
\maketitle

\section{Introduction}
A virtual contact structure which naturally appear in the study of magnetic flows
 is a generalization of a contact structure.
Dynamical properties of a virtual contact structure including 
stability, displaceability, periodic orbits, and leaf-wise intersections 
are studied in the literatures \cite{Pat, Me1, CFP, Me2, BF, Ba1}, 
when the virtual contact structure arises as an energy hypersurface of a twisted cotangent bundle.
In this article we focus on the topological aspect of a virtual contact structure.

\begin{Def}\label{def:HS}
A {\em Hamiltonian structure} on an oriented $(2n+1)$-dimensional manifold $M$
is  a closed 2-form $\om$ of maximal rank, i.e. $\om^n$ vanishes nowhere. So
\bean
\ker\om_x:=\{v\in T_xM\,|\,\iota_v\om_x=0\}
\eea
gives an 1-dimensional foliation on $M$.
By using the orientation on $M$, we orient $\ker\om$
and choose a non-vanishing  vector field $X^\om$ on $M$ 
such that $\R X^\om=\ker\om$. 
\end{Def}


\begin{Def}\label{def:vcs}
A Hamiltonian structure $(M,\om)$ is called
{\em virtual contact}
if there exists a covering $p:\widehat M\to M$ 
and a primitive $\lambda\in\Om^1(\widehat M)$ of $p^*\om$
satisfying the following conditions:
\bea\label{eqn:vircond}
\|\lambda\|_{C^0}\leq C<\infty,\qquad
\inf_{x\in\widehat M}\lambda(x)(\widehat X(x))\geq\epsilon>0
\eea
for some $C,\epsilon\in\R$.
Here $\widehat X$ is the lift of $X$ and 
$\|\cdot\|_{C^0}$-norm is given by the lifted metric $\widehat m:=p^*m$ on $\widehat M$
where $m$ is a Riemannian metric on $M$.

A virtual contact structure $(p:\widehat M\to M,\om,\lambda)$ on a smooth manifold $M$ is called {\em smooth}
if all higher covariant derivatives $\nabla^l_{\bf Y}\lambda$ are exist and uniformly bounded. 
Here ${\bf Y}$ is $l$-pairs of $G$-invariant smooth vector fields and $\nabla$ is a $G$-invariant connection on $\widehat M$,
where $G$ is the deck-transformation group of the covering $p:\widehat M\to M$.
\end{Def}

One of the main question in contact geometry is (the intrinsic version of) the Weinstein conjecture \cite{Wei}
which says that {\em every closed contact manifold $(Q,\zeta)$ has a closed orbit for any Reeb vector field}.
Indeed, we need a contact 1-form $\alpha$ for $\zeta$, i.e. $\ker\alpha=\zeta$, to define the {\em Reeb vector field} $X^\alpha$ 
which satisfies
\bean
d\alpha(X^\alpha,{_-})=0,\quad \alpha(X^\alpha)=1.
\eea
We also have a Hamiltonian structure $(Q,d\alpha)$ 
and the vector field $X^{d\alpha}$ in Definition \ref{def:HS} turns out to be a rescaling of $X^{\alpha}$. 

The Weinstein conjecture was first proved by Hofer in \cite{Hof} for a closed contact 3-manifold $M$ 
with an overtwisted disk or with $\pi_2(M)\neq0$ or $M=\S^3$.
Later, Taubes \cite{Tau} proved the conjecture for any closed contact 3-manifold, see also \cite{Hut}.
There are several extensions of the conjecture including the {\em strong Weinstein conjecture} in \cite{ACH, GZ}
and the Weinstein conjecture for the {\em stable Hamiltonian structure} in \cite{HT}.
The following question is raised by G. Paternain as another generalization of the Weinstein conjecture in the virtual contact structure.

\begin{Qu}\label{qu:vwc}
Let $(p:\widehat M\to M,\om,\lambda)$ be a virtual contact structure on a closed manifold $M$.
Does $X^\om$ admit a periodic orbit?
\end{Qu}

There is a fundamental dichotomy of contact topology on 3-manifolds, 
{\em tight} and {\em overtwisted}.
In order to state the result we need to extend these concepts to 3-dimensional virtual contact structures.
\begin{Def}\label{def:ot}
An embedded disk $F$ in a contact 3-manifold $(Q,\zeta)$ is an {\em overtwisted disk} if
$T\p F\subset\zeta|_{\p F}$ and $TF\cap\zeta|_F$ defines a smooth 1-dimensional characteristic foliation on $F$
except a unique elliptic singular point $e\in{\rm int}F$ with $T_eF=\zeta_e$.
A virtual contact structure $(p:\widehat M\to M,\om,\lambda)$ is called {\em overtwisted} if 
$(\widehat M,\ker\lambda)$ contains an overtwisted disk.
\end{Def}

Let us briefly mention Hofer's argument of deriving a contractible periodic orbit from an overtwisted disk.
By the filling method he constructed a family of pseudoholomorphic disks, a {\em Bishop family}, 
in the symplectization of a closed contact 3-manifold.
The overtwisted disk guarantees a gradient exploding sequence in the Bishop family 
and by the rescaling argument we obtain a finite energy plane.
The failure of the finite energy plane to be a sphere produce a periodic orbit  as we desired.
By extending the above argument we obtain the following result:\\[-2mm]

\noindent{\bf Main Theorem.}
{\em
Let $(p:\widehat M\to M,\om,\lambda)$ be a virtual contact structure on a closed 3-manifold $M$.
If $(p:\widehat M\to M,\om,\lambda)$ is smooth and overtwisted then $X^\om$ has a contractible periodic orbit.
}\\[-2mm]


Hofer's argument was initiated from the Eliashberg's filling technique in \cite{Eli} 
which induces an obstruction to symplectic fillings. 
A similar question about symplectic filling is possible in the virtual contact structure.
We can also expect a higher dimensional generalization of the main theorem as achieved in \cite{AlH, NR}.\\

{\em Acknowledgement} : I sincerely thank Otto van Koert for valuable and continued discussions. 
I am also grateful to Urs Frauenfelder, Peter Albers, Yong-geun Oh, and Rui Wang for their helpful comments, 
and to the organizers of East Asian Symplectic Conference in the Kagoshima University 
for giving me a chance to give a talk. 
This work was started during a stay at the University of M\"unster(SFB 878 - Groups, Geometry and Actions) 
and the paper is written in the Institute for Basic Science(the Research Center Program of IBS in Korea, CA1305) in POSTECH. 
I thank them both for their stimulating working atmosphere and generous support.

\section{An almost complex structures 
on a virtual contact structure}
Let $(p:\widehat M\to M, \om, \lambda)$ be a virtual contact structure on a $(2n+1)$-dimensional Riemannian manifold $(M,m)$,
then $\lambda\in\Om^1(\widehat M)$ is a contact form on $\widehat M$, i.e. $\lambda\wedge d\lambda^n>0$. 
Let us denote $X^\lambda$ by the Reeb vector field for the contact 1-form $\lambda$
\bean
d\lambda(X^\lambda,{_-})=0,\quad \lambda(X^\lambda)=1.
\eea
By the virtual contact condition ($\ref{eqn:vircond}$) there are positive constants $\epsilon',\,C'$ such that
$\epsilon'\leq|X^\lambda|_{\widehat m}\leq C'$.
Note that $X^\lambda$ is a rescaling of the vector field $\widehat X$ in Definition \ref{def:vcs}.
Associated to the virtual contact structure, we have the 
following canonical decomposition:
\bea\label{eqn:condec}
T\widehat M=(\ker\widehat\om,X^\lambda)\oplus(\xi:=\ker\lambda,\widehat\om).
\eea
Here $\ker\widehat\om$ is a line bundle with the section $X^\lambda$
and $(\ker\lambda,\widehat\om)$ is a symplectic bundle on $\widehat M$.

Now we choose an almost complex structure $J:\xi\to\xi$
such that
\bea\label{eqn:acs1}
\widehat m_x(\pi_\lambda k_1,\pi_\lambda k_2)
=d\lambda_x(\pi_\lambda k_1,J(x)\pi_\lambda k_2),
\eea
where $k_1,k_2\in T_x\widehat M$, $\pi_\lambda:T\widehat M\to\xi$ be the fibrewise
projection map along the Reeb direction $X^\lambda$.
Let $\widetilde J$ be the associated almost complex structure on $\R\times\widehat M$ defined by
\bea\label{eqn:acs2}
\widetilde J(a,u)(h,k):=(-\lambda(u) k,J(u) \pi_\lambda k+h\cdot X^\lambda(u)).
\eea
Now define a corresponding Riemannian metric on $\R\times\widehat M$ by
\bea\label{eqn:met}
m_{\lambda}((h_1,k_1),(h_2,k_2)):=h_1h_2+\lambda(k_1)\lambda(k_2)+\widehat m(\pi_\lambda k_1,\pi_\lambda k_2).
\eea

\begin{Rmk}\label{rmk:notjinv}
Suppose that the deck transformation group $G$ is finite, then by averaging $\lambda$ we obtain a 1-form
\bean
\overline\lambda=\frac{1}{|G|}\sum_{g\in G}g^*\lambda
\eea
 on $\widehat M$. Since $\overline\lambda$ is $G$-invariant, it descends to  a 1-form $\underline\lambda\in\Om^1(M)$
such that $\overline\lambda=p^*\underline\lambda$ and $d\underline\lambda=\om$. 
So $\underline\lambda$ becomes a contact 1-form on $M$ and hence 
the dynamics of a contact manifold $(M,\underline\lambda)$
determines the dynamics of the virtual contact structure $(p:\widehat M\to M,\om,\lambda)$.

From now on we mainly consider the case that $|G|$ is infinite\footnote{More precisely, we consider a {\em non-amenable group}
which means that there is {\em no} averaging operation on bounded functions, see \cite{CFP} for the details.}, 
i.e. $\widehat M$ is noncompact, and $\lambda$ is not $G$-invariant.
In such a case, $G$ does not preserve $X^\lambda$, $\xi$ and
hence $J$, $\widetilde J$, $m_\lambda$ are not preserved by the $G$-action,
while $\widehat\om$ and $\R X^\lambda=\ker\widehat\om$ are $G$-invariant.

By virtue of the relation (\ref{eqn:acs1}),
$J$ and $\widetilde J$ behave well under the $G$-action even though they are not $G$-invariant,
see (\ref{eqn:m_k}) in the proof of Lemma \ref{lem:Cinfty}.
If we require only the almost complex structure $J$ on $\xi$ to be $d\lambda$-compatible
instead of (\ref{eqn:acs1}), then we cannot control the limit behavior of $J$ and $\widetilde J$
on the unbounded region of $\widehat M$.

\end{Rmk}

\section{A Bishop family from an overtwisted disk}
In this section we introduce a Bishop family in the virtual contact structure 
and recall its known properties from \cite{Hof, AH, HWZ}.
In this section we provide a gradient exploding sequence of pseudoholomorphic disks
from an overtwisted disk.
All constructions, theorems, and lemmas in this section are direct consequences of the ones in the above references, 
so we omit the proof here.

Let $(p:\widehat M\to M,\om,\lambda)$ be a virtual contact structure on a smooth 3-manifold $M$
and $F\subset \widehat M$ be an overtwisted disk with the elliptic singular point $e\in{\rm int}F$.
We consider a family of pseudoholomorphic disks
$\widetilde u_\tau=(a_\tau,u_\tau):\D=\{z\in\C\,|\,|z|\leq1\}\to\R\times\widehat M,\ \tau\in[0,\tau_0)$
\bea\label{eqn:Jeqn}
\p_s\widetilde u_\tau+\widetilde J(\widetilde u_\tau)\p_t\widetilde u_\tau=0
\eea
satisfying the following conditions:
\begin{enumerate}
\item $\widetilde u_0\equiv (0,e)$;
\item $\widetilde u_\tau(\p\D)\subset\{0\}\times({\rm int}F\setminus \{e\})$ for all $\tau\in(0,\tau_0)$;
\item $\bigcup_{0\leq\tau<\tau_0}u_\tau(\p\D)$ is an open neighborhood of $e$ in $F$;
\item $\widetilde u_\tau(\D)\cap\widetilde u_\rho(\D)=\emptyset$ if $\tau\neq\rho$;
\item $\widetilde u_\tau$ is an embedding for $\tau\in(0,\tau_0)$ and $u_\tau(\p\D)$ winds once around $e$.
\end{enumerate}

Such a family of pseudo-holomorphic disks $\B$
exists and we call $\B$ a {\em Bishop family}.
We state the implicit function theorem near an embedded pseudoholomorphic disk as follows:

\begin{Thm}[\cite{Hof}]\label{thm:implicit}
Let $(p:\widehat M \to M,\om,\lambda)$ be a 3-dimensional virtual contact structure 
with an overtwisted disk $F\subset \widehat M$.
Moreover, let $\widetilde u_0$ be a pseudohomolophic disk in $(\R\times\widehat M,\widetilde J)$ 
satisfying the condition (2), (5) of the Bishop family.
Then there exists a smooth embedding $U:(-\epsilon,\epsilon)\times\D\to \R\times \widehat M$ such that 
with $\widetilde u(\tau)(z):=U(\tau,z)$ we have
\bean
\widetilde u(\tau)(z)&\in F\text{ for all }z\in\p\D;\\
\bar\p_J \widetilde u(\tau)&=0\text{ for all }\tau\in(-\epsilon,\epsilon);\\
\widetilde u(0)&=\widetilde u_0.
\eea
Moreover, the associated disk family $\tau\mapsto \widetilde u(\tau)(\D)$ is unique up to parametrization of $\D$.\footnote{
In the original statement, $(\R\times\widehat M,\{0\}\times F)$ can be generalized to an almost complex manifold with a totally real submanifold. Here the condition (2) of the Bishop family guarantees the totally real boundary condition. Then the condition (5) should be replaced by the {\em Maslov index} condition, $\mu(\widetilde u_0)=2$.}
\end{Thm}


If $\B=(\widetilde u_\tau)_{\tau\in[0,\tau_0)}$ is a Bishop family and 
$(\phi_\tau:\D\to \D)_{\tau\in[0,\tau_0)}$ is a $\tau$-parametrized family of conformal maps
then $(\widetilde u_\tau\circ\phi_\tau)_{\tau\in[0,\tau_0)}$ is also a Bishop family.
In order to fix a parametrization\footnote{Since the group of biholomorphic maps on $\D$ is 3-dimensional, 
it suffices to fix 3 boundary points.} of $\widetilde u_\tau(\D)$,
let us first parametrize the leaves of the characteristic foliation, in Definition \ref{def:ot}, approaching the singular point $e$
by $(l_\alpha)_{\alpha\in S^1}$ and require
\bea\label{eqn:normal}
u_\tau(1)\in l_1,\quad u_\tau(i)\in l_i,\quad u_\tau(-1)\in l_{-1}.
\eea
Moreover, this normalization condition (\ref{eqn:normal}) prohibits 
the existence of a gradient explosion sequence on the boundary as follows:

\begin{Thm}[\cite{HWZ}]\label{thm:bdsing}
Let $(p:\widehat M\to M,\om,\lambda)$ be a 3-dimensional virtual contact structure with an overtwisted disk $F\subset \widehat M$.
Let $\B=(\widetilde u_\tau)_{\tau\in[0,\tau_0)}$ be a Bishop family 
with the normalization condition (\ref{eqn:normal}) as in the above.
Then there exists $\epsilon>0$ such that on the annulus $A_\epsilon=\{z\in\D\,|\,1-\epsilon\leq|z|\leq 1\}$
we have
\bean
\sup_{0\leq\tau<\tau_0}\sup_{z\in A_\epsilon}|\nabla\widetilde u_\tau(z)|<\infty.
\eea
\end{Thm}

\begin{Rmk}
Even though Theorem \ref{thm:bdsing} is proved only for closed contact 3-manifolds, it is still valid for the non-compact case. Suppose that there is a gradient exploding sequence $(\widetilde u_k,z_k)$ such that $z_k$ converges to $\p\D$. 
Then $u_k(z_k)$ should converges to a point in $F$ 
and so it cannot escape to the unbounded region of $\widehat M$. 
The non-compactness of $\widehat M$ causes no additional difficulties in the proof of Theorem \ref{thm:bdsing}.
\end{Rmk}

The following observation is crucial when we produce a gradient explosion sequence 
from a given normalized Bishop family.\footnote
{Lemma \ref{lem:transv} is also essential in the proof of Theorem \ref{thm:bdsing}.
By this lemma $u_\tau|_{\p\D}$ hits each leaf of the parametrized leaves $(l_\alpha)_{\alpha\in S^1}$ exactly once.}

\begin{Lemma}[\cite{Hof}]\label{lem:transv}
Let $(p:\widehat M\to M,\om,\lambda)$ be a 3-dimensional virtual contact structure 
with an overtwisted disk $F\subset\widehat M$.
Let $\B$ be a Bishop family and take an embedded disk $\widetilde u_\tau=(a_\tau,u_\tau)\in\B$.
Then $u_\tau|_{\p\D}:\p\D\to F$ is transversal to the foliation $TF\cap\xi|_F$ on $F$. 
\end{Lemma}

By the definition of the overtwisted disk $F$, Definition \ref{def:ot}, the boundary $T\p F$ is
contained in the foliation $TF\cap\xi|_F$ on $F$.
So Lemma \ref{lem:transv} informs us that $u_\tau(\p\D)$ cannot meet $\p F$.
Since our Bishop family $\B=(\widetilde u_\tau)_{\tau\in[0,\tau_0)}$ emanated from
the singular point $e\in{\rm int}F$,  $u_\tau(\p\D)$ never touch $\p F$ nor escape it.
In other words, $e\in {\rm int}F$ enables us to create $\B$, while $\p F$ gives us an obstruction to extend $\B$.

We may assume that $\B=(\widetilde u_\tau)_{\tau\in[0,\tau_0)}$ is 
a maximal Bishop family with the normalization condition without loss of generality.
Suppose that $\|\nabla\widetilde u_\tau\|_{C^0(\D)}$ is $\tau$-uniformly bounded, 
then by the elliptic estimate of a pseudoholomorphic disk we have a $\tau$-uniform $C^\infty$-bound.
The uniform gradient bound guarantees that the image $\widetilde u_\tau(\D)$ 
is also uniformly bounded in $\R\times\widehat M$.
We then conclude by the Arzel\`a-Ascoli theorem for every sequence $\tau_k\to\tau_0$ 
there is a $C^\infty$-convergent subsequence of $\widetilde u_{\tau_k}$.
By the implicit function theorem, Theorem \ref{thm:implicit}, 
we extend our Bishop family further.
But this contradicts the maximality of our initial Bishop family $\B$.
This proves the following result.
\begin{Thm}[\cite{AH}]\label{thm:gradexp}
Let $(p:\widehat M\to M,\om,\lambda)$ be a 3-dimensional virtual contact structure 
with an overtwisted disk $F\subset\widehat M$.
Let $\B=(\widetilde u_\tau)_{0\leq\tau<\tau_0}$ be a normalized maximal Bishop family on $\R\times\widehat M$ 
which emerges from $F$.
Then we have
\bean
\sup_{\tau\in[0,\tau_0)}\|\nabla\widetilde u_\tau\|_{C^0(\D)}=\infty.
\eea
\end{Thm}

\section{Existence of a finite energy plane}
The gradient explosion in the previous section will guarantee the existence of a finite energy plane by the rescaling argument. 
The non-compactness of $\R$ and $\widehat M$ causes analytical difficulties in the rescaling process.
It is already studied that the escape phenomenon of the gradient exploding disks in $\R$-direction.
In this section we mainly discuss the analytical issue from $\widehat M$-direction.

For a $\widetilde J$-holomorphic map $\widetilde u=(a,u):\D\to\R\times\widehat M$
\bea\label{eqn:Jhol'}
\p_s\widetilde u+\widetilde J(\widetilde u)\p_t\widetilde u=0,
\eea
we define an {\em energy} $E(\widetilde u)$ by
\bea\label{eqn:energy}
E(\widetilde u):=\sup_{\varphi\in\Sigma}\int_\D\widetilde u^*d(\varphi\lambda),
\eea
where $\Sigma=\{\varphi\in C^\infty(\R,[0,1])\,|\,\varphi'\geq 0\}$.
Recall that (\ref{eqn:Jhol'}) is equivalent to
\bean
\pi_\lambda\p_s u+J(u)\pi_\lambda\p_t u&=0\\
u^*\lambda\circ i&=da
\eea
and we remark that the integrand in (\ref{eqn:energy}) is nonnegative.
By a simple computation we check that
\bea\label{eqn:dia+star}
\widetilde u^*d(\varphi\lambda)
=&\frac{1}{2}\varphi'(a)\underbrace{[a_s^2+a_t^2+(\lambda(u)u_s)^2+(\lambda(u)u_t)^2]}_{=:\diamond}ds\wedge dt\\
&+\frac{1}{2}\varphi(a)\underbrace{[|\pi_\lambda u_s|^2+|\pi_\lambda u_t|^2]}_{=:\star}ds\wedge dt\geq0.
\eea
Note that the Reeb direction of $du$ contributes to $\diamond$-term, while $\star$-term comes from the contact plane part.
By the boundary condition of the Bishop family we have the following uniform energy bound.

\begin{Lemma}[\cite{Hof}]\label{lem:ebound}
Let $F$ be an overtwisted disk as above and let $\widetilde u$ be a solution of (\ref{eqn:Jhol'}) satisfying 
the boundary condition in the definition of the Bishop family.
Then there exists a constant $C=C(\lambda,F)>0$ so that
\bean
E(\widetilde u)\leq C.
\eea
Especially $C$ does not depend on $\widetilde u$.
\end{Lemma}

Before stating the existence of a finite energy plane we introduce 
the following helpful lemma, so called Hofer's lemma, which will be used to find a suitable sequence in the rescaling argument. 
\begin{Lemma}[\cite{Hof}]\label{lem:hofer}
Let $(W,m)$ be a complete metric space and $R:W\to[0,\infty)$ a continuous function.
Assume $x_0\in W$ and $\epsilon_0>0$ are given.
Then there exist $x\in B_{2\epsilon_0}(x_0)$ and $\epsilon\in(0,\epsilon_0]$ satisfying
\bean
R(x_0)\epsilon_0&\leq R(x)\epsilon\,;\\
R(y)&\leq 2R(x)\quad\text{for }y\in B_\epsilon(x). 
\eea
\end{Lemma}

\begin{Thm}\label{thm:fep}
Let $M$ be a closed 3-manifold equipped with a smooth virtual contact structure 
$(p:\widehat M\to M,\om,\lambda)$ and an overtwisted disk $F\subset\widehat M$.
Let $\B=(\widetilde u_\tau)_{0\leq\tau<\tau_0}$ be a normalized maximal Bishop family on $(\R\times\widehat M,\widetilde J)$ 
which emerges from $F$.
Morover we have $\sup_{\tau\in[0,\tau_0)}\|\nabla\widetilde u_\tau\|=\infty.$
Then there exist an almost complex structure $\widetilde J^\infty$ on $\R\times\widehat M$
and a non-constant $\widetilde J^\infty$-holomorphic map $\widetilde v^\infty=(b^\infty,v^\infty):\C\to\R\times\widehat M$
with finite energy.\footnote{Here the energy is given by the contact form $\lambda_\infty$ in (\ref{eqn:vljinfty}).}
\end{Thm}

\begin{proof}
Since $\sup_{\tau\in[0,\tau_0)}\|\nabla\widetilde u_\tau\|=\infty$, 
we choose a sequence $(\widetilde u_k)_{k\in\N}$ from our Bishop family $\B$ 
satisfying 
\bean
\lim_{k\to\infty}\|\nabla\widetilde u_k\|=\infty.
\eea
We pick a sequence $(z_k)_{k\in\N}$ in $\D$ so that $R_k:=|\nabla\widetilde u_k(z_k)|\to\infty$ as $k\to\infty$.

If the image of the sequence $(u_k(z_k))_{k\in\N}$ is bounded in $\widehat M$
then we are able to directly apply the Hofer's argument to guarantee the existence of a finite energy plane. 
However, there is no a priori reason that $(u_k(z_k))_{k\in\N}$ is contained in a bounded region.
To remedy this situation we will use the compactness of $M$ 
via the projection $p:\widehat M\to M$.
If we consider the sequence $\underline u_k(z_k):=p\circ u_k(z_k)\in M$ 
then it has a convergent subsequence on $M$, we still denote $u_k, z_k$.
Let us fix a fundamental domain $\underline M\subset \widehat M$ 
with respect to the deck-transformation on $\widehat M$.
Now we choose a sequence of deck-transformations $g_k\in G$ such that each $g_k^{-1}\circ u_k(z_k)$
is contained in the fixed fundamental domain $\underline M$.

In order to rescale the gradient explosion, take a sequence $\epsilon_k\to0$ so that $R_k\epsilon_k\to\infty$.
Using Lemma \ref{lem:hofer}, by slightly changing $z_k$ and $\epsilon_k$, we may assume in addition that
$|\nabla\widetilde u_k(z)|\leq 2R_k$
for all $z\in\D$ with $|z-z_k|\leq\epsilon_k$.
We define a sequence of maps $\widetilde v_k : B_{R_k}(-R_k z_k)\to \R\times\widehat M$ by
\bea\label{eqn:v_k}
\widetilde v_k(z)&:=(b_k,v_k)\\
&:=\left(a_k(z_k+\frac{z}{R_k})-a_k(z_k),\, g_k^{-1}\circ u_k(z_k+\frac{z}{R_k})\right),
\eea
so that
\bea\label{eqn:0image}
b_k(0)=0,\quad v_k(0)\in\underline M
\eea
and
\bean
|\nabla\widetilde v_k(0)|=1.
\eea
Note that a sequence of domains 
\bea\label{eqn:Bk}
B^k:=B_{R_k}(-R_kz_k)\cap B_{\epsilon_k R_k}(0)
\eea
satisfies $\bigcup_{k\in\N}B^k=\C$,\footnote{Here we use Theorem \ref{thm:bdsing}. 
If the sequence $(z_k)_{k\in\N}$ converges to $\p\D$, then $\cup_{k\in\N}B^k$ could be an upper half plane.} and we have a uniform gradient bound 
\bea\label{eqn:ugbvk}
|\nabla\widetilde v_k(z)|\leq 2\quad\text{ for }z\in B^k.
\eea

Now we consider a sequence of restrictions $(\widetilde v_k|_{B^1})_{k\in\N}$.
The conditions (\ref{eqn:0image}), (\ref{eqn:ugbvk}) imply that
the image $\widetilde v_k(B^1)$ is uniformly bounded in $\R\times\widehat M$.
By applying the Arzel\`a-Ascoli theorem we have a subsequence, again denote $\widetilde v_{k}|_{B^1}$,
and a continuous map $\widetilde v^1:B^1\to\R\times\widehat M$ such that
$\widetilde v_k|_{B^1}$ converges uniformly to $\widetilde v_1$.
Recall from Remark \ref{rmk:notjinv} that the almost complex structure $\widetilde J$ on $\R\times\widehat M$ 
is not invariant under the $G$-action, and hence
$g_k^{-1} \widetilde u_k:=(a_k,g_k^{-1}\circ u_k),\ \widetilde v_k$, and $\widetilde v^1$ cannot be $\widetilde J$-holomorphic.
So we need to find new almost complex structures which make $\widetilde v_k$, $\widetilde v^1$ to be pseudo-holomorphic.

First choose a compact subset $E^1\subset\widehat M$ 
containing $\bigcup_{k}v_{k}(B^1(0))$.
Let us define a sequence of almost complex structures $(\widetilde J_k)_{k\in\N}$ on $\R\times\widehat M$ by
\bean
\widetilde J_k(a,u)(h,l):=dg^{-1}_k[\widetilde J(a,g_k u)(h,dg_k l)],
\eea
then $g_k^{-1}\widetilde u_k$ becomes $\widetilde J_k$-holomorphic
and so does $\widetilde v_k$.
In order to understand the sequence of almost complex structures $(\widetilde J_k)_{k\in\N}$ and their limits
we consider the following sequences
\bean
\lambda_k:=g_k^*\lambda,\quad
\xi_k:=\ker \lambda_k,\quad
X_k:=(g_k^{-1})_*X^\lambda,\quad 
J_k:=\widetilde J_k|_\xi.
\eea
because of the construction of $\widetilde J$, (\ref{eqn:acs1}) and (\ref{eqn:acs2}).
More precisely, we have
\bea\label{eqn:krel}
\lambda_k(x)v&= \lambda (g_k x)dg_k v;\\
\xi_k(x)&=\{dg_k^{-1}w\,|\,w\in\xi(g_k x)\};\\
X_k(x)&=dg^{-1}_kX^\lambda(g_k x);\\
J_k(x)v&=dg_k^{-1}[J(g_k x)dg_k v],
\eea
where $x\in\widehat M$ and $v\in T_x\widehat M$.
Note that $X_k$ is a Reeb vector field of $\lambda_k$ and $J_k$ is an almost complex structure on $\xi_k$.

For convenience, we choose a smooth coframe field $\{\rho_1,\rho_2,\rho_3\}$ of the closed 3-manifold $M$.\footnote{Here we use that the tangent bundle $TQ$ of any closed 3-manifold $Q$ is trivial.}
Then its lift $\{\widehat\rho_1,\widehat\rho_2,\widehat\rho_3\}$ gives an induced coframe on $\widehat M$. 
There are coefficient functions $c_1,c_2,c_3\in C^\infty(\widehat M)$ for $\lambda\in\Om^1(\widehat M)$ satisfying
\bean
\lambda=c_1\widehat\rho_1+c_2\widehat\rho_2+c_3\widehat\rho_3.
\eea
By the smooth condition for the virtual contact structure in Definition \ref{def:vcs} 
and the compactness of $M$, 
all higher directional derivatives $\nabla^l_{\bf Y}c_i$, $i=1,2,3$ are exist and uniformly bounded.
Here ${\bf Y}$ is again $l$-pairs of $G$-invariant smooth vector fields.
Note that the following simple observation
\bean
\nabla^l_{\bf Y}(c_i\circ g_k)(x)=\nabla^l_{dg_k({\bf Y})}c_i(g_k(x))=\nabla^l_{\bf Y}c_i(g_k(x))
\eea
holds for $x\in\widehat M$, $g_k\in G$, and $i=1,2,3$.
So all higher directional derivatives of $c_i\circ g_k$ are also exist and uniformly bounded.
Now we apply the Arzel\`a-Ascoli theorem to the sequence of 1-forms
\bean
\lambda_k|_{E^1}=\sum_{i=1}^3(c_i\circ g_k)\widehat \rho_i|_{E^1},
\eea
in order to obtain a subsequence, still denote $\lambda_k$,
which converges to $\lambda^1\in\Om^1(E^1)$ in $C^\infty$.


\begin{Lemma}\label{lem:lambda1}
The limit 1-form $\lambda^1$ is a contact 1-form on int$E^1$.
\end{Lemma}
\begin{proof}
Since $\lambda$ is a contact 1-form on $\widehat M$, $\lambda_{k}=g_k^*\lambda$ satisfies
\bean
\lambda_{k}\wedge d\lambda_{k}>0\quad\text{ on int}E^1
\eea
for each $k\in\N$. As a limit of $\lambda_k$, however, $\lambda^1$ may degenerate and could satisfy 
$\lambda^1\wedge d\lambda^1\geq0$ so it is needed to exclude the case $\lambda^1\wedge d\lambda^1=0$.
Suppose that there is $x\in {\rm int}E^1$ such that $\lambda^1(x)\wedge d\lambda^1(x)=0$.
By the construction of $\lambda^1$, 
its exterior derivative $d\lambda_k|_{E^1}$ also converges to $d\lambda^1$.
Thus we have
\bean
d\lambda_k|_{E^1}=dg_k^*\lambda|_{E^1}=g_k^*d\lambda|_{E^1}=g_k^*\widehat\om|_{E^1}=\widehat \om|_{E^1}
\longrightarrow d\lambda^1,
\eea
which implies $\lambda^1(x)\wedge\widehat\om(x)=0$.
Let us recall the decomposition $T_x\widehat M\cong\R\widehat X(x)\oplus\ker\lambda(x)$
from (\ref{eqn:condec}) where $\widehat X(x)$ is defined in Definition \ref{def:vcs} 
which generates $\ker\widehat\om(x)$ and $G$-invariant.
So we deduce
\bean
0&=\iota_{\widehat X(x)}(\lambda^1(x)\wedge\widehat\om(x)) \\
&=\lambda^1(x)(\widehat X(x))\cdot\widehat\om(x)+\lambda^1(x)\wedge\iota_{\widehat X(x)}\widehat\om(x) \\
&=\lambda^1(x)(\widehat X(x))\cdot\widehat\om(x).
\eea
Since $\widehat\om(x)\neq0$ on $\ker\lambda(x)$,
we deduce $\lambda^1(x)(\widehat X(x))=0$.
From the definition of $\lambda^1, \lambda_i$ we obtain
\bean
0&=\lambda^1(x)(\widehat X(x))\\
&=\lim_{k\to\infty}\lambda_{k}(x)(\widehat X(x))\\
&=\lim_{k\to\infty}(g_{k}^*\lambda)(x)(\widehat X(x))\\
&=\lim_{k\to\infty}\lambda(g_k x)(dg_k(\widehat Xx))\\
&=\lim_{k\to\infty}\lambda(g_{k}x)(\widehat X(g_{k}x)),
\eea
where the last equality comes from that $\widehat X$ is $G$-invariant.
But this cannot be possible because of the virtual contact condition in Definition \ref{def:vcs}
\bean
\inf_{x\in\widehat M}\lambda(x)(\widehat X(x))\geq\epsilon>0.
\eea
\end{proof}
By this contact 1-form $\lambda^1$ 
we can construct the corresponding contact structure $\xi^1:=\ker\lambda^1$, 
the Reeb vector field $X^1$ for $\lambda^1$ on $E^1$ 
with the decomposition $TE^1=\R X^1\oplus\xi^1$ and the projection $\pi^1:TE^1\to\xi^1$.
By the same construction we define
the almost complex structures $J^1,\widetilde J^1$ on $\xi^1,\, T(\R\times E^1)$ by satisfying
\bea\label{eqn:lambdaJ}
\widehat m_x(\pi^1 k_1,\pi^1 k_2)&=\widehat\om_x(\pi^1 k_1,J^1(x)\pi^1 k_2),\\
\widetilde J^1(a,u)(h,k)&=(-\lambda^1(u)k,J^1(u)\pi^1 k+h\cdot X^1(u)).
\eea 
Note here that the vector field $X^1$ is different from $X^\lambda|_{E^1}$ but both vector fields generate $\ker\widehat\om|_{E^1}$.

\begin{Lemma}\label{lem:Cinfty}
As in the above setting, $\lambda_{k}|_{E^1}\stackrel{C^\infty}{\longrightarrow}\lambda^1$ implies
$\widetilde J_{k}|_{E^1}\stackrel{C^\infty}{\longrightarrow}\widetilde J^1$.
\end{Lemma}

\begin{proof}
Let us define $\mu_k:=\lambda_k|_{E^1}-\lambda^1\in\Om^1(E^1)$ which converges to $0$ in $C^\infty$ and
recall that $X_k$ the Reeb vector field of $\lambda_k$.
Since $d\lambda_k |_{E^1}=\widehat\om |_{E^1}=d\lambda^1$, 
we infer $\R X_k=\ker\widehat\om=\R X^1$ and hence
$X_k|_{E^1}$ is a rescaling vector field of $X^1$.
More precisely,
\bean
X_k|_{E^1}=\frac{1}{1+\mu_k(X^1)}X^1.
\eea
Here $X^1$ is a bounded smooth vector field on $E^1$ and hence 
$X_k|_{E^1}$ also $C^\infty$-converges to $X^1$.
Note that $\pi_{k}:TE^1\to\xi_k|_{E^1}$ is equal to $\one_{TE^1}-\lambda_{k}(_-)X_k$
and $\pi^1=\one_{TE^1}-\lambda^1(_-)X^1$.
Since
\bea\label{eqn:LRcon}
\lambda_{k}|_{E^1}\stackrel{C^\infty}{\longrightarrow}\lambda^1, \qquad
X_{k}|_{E^1}\stackrel{C^\infty}{\longrightarrow}X^1,
\eea
$\pi_k |_{E^1}$ converges to $\pi^1$ in $C^\infty$.
Now consider a sequence of metrics $\big(\widehat m(\pi_{k}{_-} ,\pi_{k}{_-})\big)_{k\in\N}$ on $\xi_k$ then
we subsequently have 
\bea\label{eqn:mconv}
\widehat m(\pi_{k}{_-} ,\pi_{k}{_-})\stackrel{C^\infty}{\longrightarrow}\widehat m(\pi^1{_-},\pi^1{_-})
\eea
as symmetric bilinear forms on $E^1$.

We now define a metric $m_k$ on $\xi_k$ by
\bean
(m_k)_x(\pi_k h,\pi_k l):=\widehat\om_x(\pi_k h,J_k(x)\pi_k l),
\eea
where $\widetilde J_k$ is the almost complex structure in (\ref{eqn:krel}).
For $h,l\in TE^1$ we then have
\bea\label{eqn:m_k}
(m_k)_x(\pi_k h,\pi_k l)&=\widehat\om_x(\pi_k h,J_k(x)\pi_k l)\\
&=\widehat\om_x(\pi_k h,dg_k^{-1}J(g_k(x))dg_k\pi_k l)\\
&=(g_k^*\widehat\om)_x(\pi_k h,dg_k^{-1}J(g_k(x))dg_k\pi_k l)\\
&=\widehat\om_{g_k(x)}(dg_k\pi_k h,J(g_k(x))dg_k\pi_k l)\\
&=\widehat m_{g_k(x)}(dg_k\pi_k h,dg_k\pi_k l)\\
&=(g_k^*\widehat m)_x(\pi_k h,\pi_k l)\\
&=\widehat m_x(\pi_k h,\pi_k l).
\eea
Here the 3rd and the last equality come from the $G$-invariance of $\widehat\om$, $\widehat m$ respectively.\footnote
{Suppose that our metric $\widehat m$ is {\em not} $G$-invariant then $\lim_{k\to\infty}g_k^*\widehat m$ may not be a metric anymore.
Moreover, in such a case, we can not define $J^1$ as in (\ref{eqn:lambdaJ}).}
By combining (\ref{eqn:lambdaJ}), (\ref{eqn:mconv}), (\ref{eqn:m_k}) we deduce
\bean
\widehat\om(\pi_k{_-},J_k\pi_k{_-})\stackrel{C^\infty}{\longrightarrow}\widehat\om(\pi^1{_-},J^1\pi^1{_-})
\eea
as 2-forms on $E^1$ and hence
\bea\label{eqn:J_kJ^1}
J_k\circ\pi_k|_{E^1}\stackrel{C^\infty}{\longrightarrow}J^1\circ\pi^1
\eea
as $(1,1)$-forms on $E^1$.

Now we are ready to compare the almost complex structures $\widetilde J_k|_{E^1}$ and $\widetilde J^1$
\bea\label{eqn:tJcon}
\widetilde J_{k}(a,u)(h,l)&=(-\lambda_{k}(u)l,J_{k}(u)\pi_k l+h\cdot X_{k}(u)),\\
\widetilde J^1(a,u)(h,l)&=(-\lambda^1(u)l,J^1(u)\pi^1 l+h\cdot X^1(u))
\eea
where $(a,u)\in \R\times E^1$, $(h,l)\in T_{(a,u)}(\R\times E^1)$.
By (\ref{eqn:LRcon}), (\ref{eqn:J_kJ^1}) and (\ref{eqn:tJcon}) we finally conclude that
$\widetilde J_k|_{E^1}$ converges to $\widetilde J^1$ in $C^\infty$-topology.
\end{proof}

Up to now, we have a sequence of triples $(v_{k_1}|_{B^1},\lambda_{k_1}|_{E^1},\widetilde J_{k_1}|_{\R\times E^1})_{k_1\in\N}$
and $(v^1,\lambda^1,\widetilde J^1)$ with the following convergence:
\bean
v_{k_1}|_{B^1}\stackrel{C^0}{\longrightarrow} v^1,\quad
\lambda_{k_1}|_{E^1}\stackrel{C^\infty}{\longrightarrow}\lambda^1,\quad
\widetilde J_{k_1}|_{\R\times E^1}\stackrel{C^\infty}{\longrightarrow}\widetilde J^1.
\eea
Let us recall the sequence of bounded domains 
$B^n=B_{R_n}(-R_n z_n)\cap B_{\epsilon_n R_n}(0)$ from (\ref{eqn:Bk})
satisfying $\bigcup_{n\in\N}B^n=\C$
and consider a sequence of compact subsets $E^n\subset \widehat M$ satisfying
\begin{itemize}
\item $E^n\subset E^{n+1}$ for all $n\in\N$;
\item $\bigcup_{k\in\N} v_k(B^n)\subset E^n$ for all $n\in\N$;
\item $\bigcup_{n\in\N}E^n=\widehat M$.
\end{itemize}

Now we apply the Arzel\`a-Ascoli theorem to the triple $(v_{k_1},\lambda_{k_1},\widetilde J_{k_1})_{k_1\in\N}$ inductively.
For $n\geq2$, we pick a subsequence\footnote{Here taking a subsequence is equivalent to choose an increasing and unbounded function $s_n:\N\to\N$ satisfying $k_n=k_{n-1}\circ s_n:\N\to\N$. Note that $k_1:\N\to\N$ is the identity map.}
$k_n$ of $k_{n-1}$ such that
there exists a continuous map $v^n:B^n\to\R\times E^n$ with
\bean
v_{k_n(j)}|_{B^n}\stackrel{C^0}{\longrightarrow} v^n\quad\text{ as } j\to\infty.
\eea
By the same argument as in Lemma \ref{lem:lambda1}, \ref{lem:Cinfty}
we obtain a contact form $\lambda^n\in\Om^1(E^n)$ and
an almost complex structure $\widetilde J^n$ on $\R\times E^n$ satisfying
\bean
\lambda_{k_n(j)}|_{E^n}\stackrel{C^\infty}{\longrightarrow}\lambda^n,\quad
\widetilde J_{k_n(j)}|_{\R\times E^n}\stackrel{C^\infty}{\longrightarrow}\widetilde J^n
\quad\text{ as } j\to\infty.
\eea
As a consequence of the above construction, we have
\bean
\widetilde v^{n+1}|_{B^n}=\widetilde v^n,\quad
\lambda^{n+1}|_{E^n}=\lambda^n,\quad
\widetilde J^{n+1}|_{\R\times E^n}=\widetilde J^n\quad\text{ for all }n\in\N.
\eea
Especially note that $\lambda^n$ determines $\widetilde J^n$ as in (\ref{eqn:lambdaJ}) and
$v^n$ is $\widetilde J^n$-holomorphic for all $n\in\N$.

We now consider the diagonal sequence of triples 
$(\widetilde v_{k_j(j)}|_{B^j},\lambda_{k_j(j)}|_{E^j},\widetilde J_{k_j(j)}|_{\R\times E^j})_{j\in\N}$,
then we obtain a continuous map $\widetilde v^\infty:\C\to\R\times\widehat M$, 
a contact form $\lambda^\infty\in\Om^1(\widehat M)$ and
an almost complex structure $\widetilde J^\infty$ on $\R\times\widehat M$
satisfying
\bea\label{eqn:vljinfty}
\widetilde v_{k_j(j)}|_{B^j}\stackrel{C^0_{loc}}{\longrightarrow} \widetilde v^\infty, \quad
\lambda_{k_j(j)}|_{E^j}\stackrel{C^\infty_{loc}}{\longrightarrow}\lambda^\infty,\quad
\widetilde J_{k_j(j)}|_{\R\times E^j}\stackrel{C^\infty_{loc}}{\longrightarrow}\widetilde J^\infty
\quad\text{ as } j\to\infty.
\eea
and
\bean
\widetilde v^{\infty}|_{B^n}=\widetilde v^n,\quad
\lambda^{\infty}|_{E^n}=\lambda^n,\quad
\widetilde J^{\infty}|_{\R\times E^n}=\widetilde J^n\quad\text{ for all }n\in\N.
\eea
Moreover, the limit 1-form $\lambda^\infty$ defines a new virtual contact structure as follows:
\begin{Lemma}\label{lem:newvcs}
As in the above setting. Let $(p:\widehat M\to M,\om,\lambda)$ be a smooth virtual contact structure,
then so is $(p:\widehat M\to M,\om,\lambda^\infty)$.
\end{Lemma}

\begin{proof}
Since $\lambda_k$ converges to $\lambda^\infty$ in $C^\infty_{loc}$, 
its exterior derivative $d\lambda_k$ also converges to $d\lambda^\infty$.
This implies that $\lambda^\infty$ is again a primitive of $\widehat \om$.
By the construction of $\lambda^\infty$ there is a sequence $(g_i)_{i\in\N}$ of deck transformations which satisfies
the following estimate for any $x\in\widehat M$, any $l\in\N$, and 
$l$-pairs of $G$-invariant smooth vector fields ${\bf Y}=(Y_1,Y_2,\dots,Y_l)$:
\bea\label{eqn:linfty1}
|\nabla^l_{\bf Y}\lambda^\infty(x)|_{\widehat m}
&=\max_{|v|_{\widehat m}=1}\lim_{i\to\infty}|\nabla_{\bf Y}^l g_i^*\lambda(x)v|\\
&=\max_{|v|_{\widehat m}=1}\lim_{i\to\infty}|g_i^*\nabla_{g_i\bf Y}^l \lambda(x)v|\\
&=\max_{|v|_{\widehat m}=1}\lim_{i\to\infty}|\nabla_{g_i\bf Y}^l \lambda(g_i x)dg_i v|\\
&=\max_{|v|_{\widehat m}=1}\lim_{i\to\infty}|\nabla_{g_i\bf Y}^l\lambda(g_i x)v|\\
&=\lim_{i\to\infty}|\nabla_{g_i\bf Y}^l \lambda(g_i x)|_{\widehat m}\\
&\leq\sup_{x\in\widehat M}|\nabla_{\bf Y}^l \lambda(x)|_{\widehat m}\\
&\leq C,
\eea
where $g_i{\bf Y}=(dg_iY_1,dg_iY_2,\dots,dg_iY_l)$. Here the 2nd, 4th, and 6th (in)equality come from the $G$-invariance of 
$\nabla$, $\widehat m$, and ${\bf Y}$ respectively and the last inequality is induced by
the smooth condition in Definition \ref{def:vcs}.
By the similar argument as in Lemma \ref{lem:lambda1} we have
\bea\label{eqn:linfty2}
\lambda^\infty(x)(\widehat X(x))
&=\lim_{i\to\infty}g_i^*\lambda(x)(\widehat X(x))\\
&=\lim_{i\to\infty}\lambda(g_i x)(dg_i(\widehat Xx))\\
&=\lim_{i\to\infty}\lambda(g_i x)(\widehat X(g_i x))\\
&\geq\inf_{x\in\widehat M}\lambda(x)(\widehat X(x))\\
&\geq\epsilon.
\eea
The above two estimates (\ref{eqn:linfty1}), (\ref{eqn:linfty2}) show that the virtual contact structure $(p:\widehat M\to M,\om,\lambda^\infty)$ is smooth.
\end{proof}




Therefore we have a continuous $\widetilde J^\infty$-holomorphic map $\widetilde v^\infty:\C\to\R\times\widehat M$.
The $C_{loc}^0$-convergence of $\widetilde v_{k_k}|_{B^k}$ to $\widetilde v^\infty$ can be improved to
$C^\infty_{loc}$-convergence 
by applying the elliptic bootstrapping argument, see \cite[Theorem B.4.2]{MS}.
Finally we have a smooth $\widetilde J^\infty$-holomorphic plane
\bean
\widetilde v^\infty=(b^\infty,v^\infty):\C\to\R\times\widehat M
\eea
which is non-constant in view of $|\nabla\widetilde v^\infty(0)|=1$.

Now it remains to show the finiteness of the energy 
\bean
E(\widetilde v^\infty)=\sup_{\varphi\in\Sigma}\int_\C\widetilde v^{\infty *}d(\varphi\lambda^\infty),
\eea
where $\Sigma=\{\varphi\in C^\infty(\R,[0,1])\,|\,\varphi'\geq0\}$.
For any compact set $K\subset\C$ there is $j\in\N$ sufficiently large so that $K\subset B^j$.
Then we have
\bean
\sup_{\varphi\in\Sigma}\int_K\widetilde v_j^* d(\varphi\lambda_j)
&\leq \sup_{\varphi\in\Sigma}\int_{B^j}\widetilde v_j^* d(\varphi\lambda_j)\\
&\leq \sup_{\varphi\in\Sigma}\int_{\D}\widetilde u_j^* d(\varphi\lambda)\\
&=E(\widetilde u_j).
\eea
Let $j\to\infty$ so that $\widetilde v_j$, $\lambda_j$ converge to $\widetilde v^\infty$, $\lambda^\infty$ in $C^\infty_{loc}$
and take the supremum over all compact set $K\subset\C$ then we obtain
\bean
E(\widetilde v^\infty)\leq& \lim_{j\to\infty}E(\widetilde u_j)\leq C,
\eea
where the constant $C$ comes from Lemma \ref{lem:ebound}. 

\end{proof}

\section{From a finite energy plane to a periodic orbit}
We will use the following notations for simplicity:
\bean
\alpha:=\lambda^\infty, \quad
\zeta:=\xi^\infty,\quad
\pi_\alpha:=\pi^\infty,\quad
I:=J^\infty,\quad
\widetilde I:=\widetilde J^\infty,\quad
b:=b^\infty,\quad
v:=v^\infty,\quad
\widetilde v:=\widetilde v^\infty.
\eea 
Lemma \ref{lem:newvcs} says that
$(p:\widehat M\to M, \om,\alpha)$ is a smooth virtual contact structure
with the decomposition $T\widehat M=\R \widehat X\oplus\zeta$,
 the projection $\pi_\alpha:T\widehat M\to\zeta$ along the Reeb direction $\widehat X$,
and the \acs s $I$, $\widetilde I$ on $\zeta$, $\R\times\widehat M$.
Moreover, Theorem \ref{thm:fep} implies that
$\widetilde v:\C\to\R\times\widehat M$ is a non-constant $\widetilde I$-holomorphic plane
with finite energy.
In other words, $\widetilde v=(b,v)$ is a solution of
\bea\label{eqn:J-hol}
\pi_\alpha\p_s v+I(v)\pi_\alpha \p_t v&=0\\
(v^*\alpha)\circ i&=db,
\eea
and
\bean
0<E(\widetilde v)=\sup_{\varphi\in\Sigma}\int_{\C}\widetilde v^*d(\varphi\alpha)<\infty,
\eea
where $\Sigma=\{\varphi\in C^\infty(\R,[0,1])\,|\,\varphi'\geq 0\}$.
Note here that $M$ is compact and hence the projection 
$\underline v:=p\circ v:\C\to M$ has a compact image,
while $v(\C)$ maybe non-compact.

The main aim of this section is to find a periodic orbit from the above finite energy plane.

\begin{Thm}\label{thm:finmain}
Let $(p:\widehat M\to M,\om,\alpha)$ be a smooth virtual contact structure
and $\widetilde v=(b,v):\C\to\R\times\widehat M$ be a solution of (\ref{eqn:J-hol}) satisfying
\bean
0<E(\widetilde v)<\infty,\quad \int_{\C}v^*d\alpha>0
\eea
then for every sequence $R_k\to\infty$ there exists a subsequence $(R_{k'})_{k'\in\N}$ such that
the $C^\infty$-limit 
\bean
x(t):=\lim_{k'\to\infty}\underline v(R_{k'}e^{2\pi it})
\eea
exists and its projection $\underline x(t)=p\circ x(t)$ defines a non-constant closed periodic solution of
\bean
\dot{\underline{x}}(t)=X(\underline x(t))
\eea
where $X$ is a non-vanishing vector field generating $\ker\om$.
\end{Thm}

\begin{Prop}[\cite{Hof}]\label{prop:const}
Let $\widetilde v=(b,v):\C\to\R\times\widehat M$ solves (\ref{eqn:J-hol}) with finite energy. If
\bea\label{eqn:da=0}
\int_{\C}v^*d\alpha=0,
\eea
then $\widetilde v$ is constant.
\end{Prop}

\begin{Rmk}\label{rmk:consthol}
The assumption (\ref{eqn:da=0}) implies that the energy from the contact plane, like $\star$-term in (\ref{eqn:dia+star}), vanishes
and then $b:\C\to\R$ becomes a harmonic function because of the $I$-holomorphic equation (\ref{eqn:J-hol}).
Indeed, $b$ can be regarded as a real part of the holomorphic function $\Psi:=b+i\beta:\C\to\C$ where $\beta:\C\to\R$ is a primitive of $v^*\alpha\in\Om^1(\C)$.
Now suppose that $\widetilde v$ is non-constant then so is $b$. 
Essentially, Liouville's theorem for $\Psi$ implies that the energy $E(\widetilde v)$ is infinite. This is a contradiction.
The above argument is still valid when $\widehat M$ is non-compact and we omit the detailed proof.
\end{Rmk}

Let $\phi:\R\times S^1\to\C\setminus\{0\}$ be a holomorphic map defined by
\bean
\phi(s,t)=e^{2\pi(s+it)},
\eea
where $S^1=\R/\Z$. For the later purpose we consider 
a $\widetilde I$-holomorphic cylinder instead of a plane.
So we define 
\bea\label{eqn:defofv}
\widetilde v_\phi:=(b_\phi,v_\phi):=(b,v)\circ\phi:\R\times S^1\to\R\times \widehat M
\eea
then we have
\bea\label{eqn:J-cyl}
\p_s\widetilde v_\phi+\widetilde I(\widetilde v_\phi)\p_t\widetilde v_\phi&=0;\\
\int_{\R\times S^1}v_\phi^*d\alpha&>0;\\
0<E(\widetilde v_\phi)=\sup_{\varphi\in\Sigma}\int_{\R\times S^1}\widetilde v_\phi^*d(\varphi\alpha)&<\infty,
\eea 
where $\Sigma$ is as before.

\begin{Prop}\label{prop:gradbdv}
Let $\widetilde v_\phi$ be a solution of (\ref{eqn:J-cyl}), then there exists some constant $l>0$ such that
\bean
|\nabla\widetilde v_\phi(s,t)|\leq l
\eea
for all $(s,t)\in\R\times S^1$.
\end{Prop}

\begin{proof}
Let $\rho:\C\to\R\times S^1:(s+it)\mapsto(s,e^{2\pi it})$ be a 1-periodic map in the $S^1$-coordinate
and let us define
\bean
\widetilde v_\rho:=(b_\rho,v_\rho):=(b_\phi,v_\phi)\circ\rho:\C\to\R\times\widehat M.
\eea
So it is equivalent to show that $|\nabla\widetilde v_\rho(s,t)|$ is bounded.
Suppose that there is a sequence $z_k\in \C$ such that
\bean
R_k:=|\nabla\widetilde v_\rho(z_k)|\to\infty.
\eea
Note that ${\rm Re}(z_k)\to\infty$ since the gradient is bounded on ${\rm Re}(z_k)\leq0$.
By applying Lemma \ref{lem:hofer} to positive real sequences $(R_k)_{k\in\N}$, $(\epsilon_k)_{k\in\N}$ with
\bean
\epsilon_k\to 0,\quad \epsilon_k R_k\to\infty,
\eea
we additionally assume that
\bea\label{eqn:grad2}
|\nabla\widetilde v_\rho(z)|\leq 2R_k\quad\text{for }z\in B_{\epsilon_k}(z_k)\subset\C.
\eea
The sequence $(v_\rho(z_k))_{k\in\N}$ may escape to the unbounded region of $\widehat M$. So we pick a sequence of deck-transformations $(f_k)_{k\in\N}$ in $G$ so that
\bean
f_k^{-1}\circ v_\rho(z_k)\in\underline M
\eea
where $\underline M$ is a fixed fundamental domain in $\widehat M$ with respect to $G$.
Now we define $\widetilde \mu_k:B_{R_k}(-R_k z_k)\to \R\times\widehat M$ by
\bea\label{eqn:v_kres}
\widetilde \mu_k(z)&:=(\beta_k(z),\mu_k(z))\\
&:=\left(b_\rho(\frac{z}{R_k}+z_k)-b_\rho(z_k),f_k^{-1}\circ v_\rho(\frac{z}{R_k}+z_k) \right).
\eea
We also define a sequence of contact forms $\alpha_k:=f_k^*\alpha$ and
a sequence of \acs s 
\bean
\widetilde I_k(a,u)(r,l):=df_k^{-1}[\widetilde I(a,f_k u)(r,df_k l)].
\eea
By the similar argument as in Theorem \ref{thm:fep}, there are a suitable subsequence $(k')_{k'\in\N}$, 
a smooth map $\widetilde \mu':\C\to\R\times\widehat M$, 
a contact form $\alpha'\in\Om^1(\widehat M)$ and
an \acs \  $\widetilde I'$ on $R\times\widehat M$ satisfying
\bean
\widetilde \mu_{k'}\stackrel{C^\infty_{loc}}{\longrightarrow} \widetilde \mu', \quad
\alpha_{k'}\stackrel{C^\infty_{loc}}{\longrightarrow}\alpha',\quad
\widetilde I_{k'}\stackrel{C^\infty_{loc}}{\longrightarrow}\widetilde I'.
\eea
Moreover, the limit contact form $\alpha'$ determines the limit \acs \ $\widetilde I'$ as in Lemma \ref{lem:Cinfty}
and $\widetilde \mu'$ is $\widetilde I'$-holomorphic, i.e.
\bean
\p_s\widetilde \mu'+\widetilde I'(\widetilde \mu')\p_t\widetilde \mu'=0.
\eea
From (\ref{eqn:v_kres}), (\ref{eqn:grad2}) we deduce
\bean
|\nabla\widetilde \mu'(0)|=1,\quad |\nabla\widetilde \mu'(z)|\leq 2\quad\text{for }z\in\C.
\eea
With a sequence of functions $\varphi_k(s):=\varphi(s-b(z_k))$ in $\Sigma$, we estimate
\bean
\int_{B_{\epsilon_k R_k}(0)}\widetilde \mu_k^* d(\varphi\alpha_k)
=\int_{B_{\epsilon_k}(z_k)}\widetilde v_\rho^* d(\varphi_k \alpha)
\leq\int_{\R\times[0,1]}\widetilde v_\rho^* d(\varphi_k \alpha)
\leq E(\widetilde v_\phi)
<\infty.
\eea
Replace $k$ with $k'$ and let $k'\to\infty$ then we deduce
\bean
E(\widetilde \mu')\leq E(\widetilde v_\phi)<\infty.
\eea

We know 
\bea\label{eqn:T}
T:=\int_{\R\times S^1}v_\phi^*d\alpha
=\int_{\R\times S^1}\widetilde v_\phi^*d(\varphi_0\alpha)
\leq E(\widetilde v_\phi)<\infty
\eea
where $\varphi_0\equiv 1$.
If $q:=\int_\C \mu'^*d\alpha'$ is positive, then there is a subsequence $l$ such that
\bean
\frac{q}{2}\leq\int_{B_{\epsilon_{l} R_{l}}(0)}\mu_l^*d\alpha_l=\int_{B_{\epsilon_{l}}(z_{l})}v^*d\alpha,
\eea
$\epsilon_l\leq\frac{1}{2}$, and $B_{\epsilon_{l}}(z_{l})$ are disjoint.
So we obtain a following contradiction
\bean
\infty=\sum_{l}\frac{q}{2}
\leq\sum_{l}\int_{B_{\epsilon_{l}}(z_{l})}v_\rho^*d\alpha 
\leq \int_{\R\times S^1}v_\phi^*d\alpha<\infty.
\eea
Consequently we have a non-constant $\widetilde I'$-holomorphic map $\widetilde \mu':\C\to\R\times\widehat M$ 
with a finite energy and $\int_\C \mu'^*d\alpha'=0$.
By Proposition \ref{prop:const} such a map $\widetilde \mu'$ cannot be possible.
Therefore $|\nabla\widetilde v_\phi(s,t)|$ is uniformly bounded.
\end{proof}

\begin{proof}[Proof of Theorem \ref{thm:finmain}]
Let $\widetilde v_\phi=(b_\phi,v_\phi):\R\times S^1\to\R\times\widehat M$ 
be the map from the previous proposition satisfying (\ref{eqn:J-cyl}).
We pick a sequence of real numbers $(s_k)_{k\in\N}$ with $\lim_{k\to\infty}s_k=\infty$ and 
note that $v_\phi(s_k,0)\in\widehat M$ may escape to the unbounded region as $k\to\infty$.
For the fixed fundamental region $\underline M\subset\widehat M$,
we choose a sequence of deck-transformations  $(h_k)_{k\in\N}$ so that
\bean
h_k^{-1}\circ v_\phi(s_k,0)\in\underline M.
\eea
Now we define a sequence of cylinder maps $\widetilde w_k:=(c_k,w_k):\R\times S^1\to\R\times\widehat M$ by
\bean
\widetilde w_k(s,t):=\left(b_\phi(s+s_k,t)-b_\phi(s_k,0),h_k^{-1}\circ v_\phi(s+s_k,t)\right).
\eea
Let us also define a corresponding sequence of contact forms $\alpha_k:=h_k^*\alpha$ and
a sequence of \acs s
\bean
\widetilde I_k(a,u)(r,l):=dh_k^{-1}[\widetilde I(a,h_k u)(r,dh_k l)].
\eea
By the similar procedure as in Theorem \ref{thm:fep} and Proposition \ref{prop:gradbdv} 
we choose a suitable subsequence $(k')_{k'\in\N}$ so that
there exist a smooth map
\bean
\widetilde w=(c,w):\R\times S^1\to\R\times\widehat M,
\eea
a contact form $\alpha^\infty\in\Om^1(\widehat M)$, and
an almost complex structure $\widetilde I^\infty$ on $\R\times\widehat M$ 
satisfying
\bea\label{eqn:wconv}
\widetilde w_{k'}\stackrel{C^\infty_{loc}}{\longrightarrow} \widetilde w, \quad
\alpha_{k'}\stackrel{C^\infty_{loc}}{\longrightarrow}\alpha^\infty,\quad
\widetilde I_{k'}\stackrel{C^\infty_{loc}}{\longrightarrow}\widetilde I^\infty.
\eea
In addition, the contact form $\alpha^\infty$ governs the almost complex structure $\widetilde I^\infty$
in the view of Lemma \ref{lem:Cinfty} and $\widetilde w$ is an $\widetilde I^\infty$-holomorphic cylinder, i.e.
\bea\label{eqn:Ihol}
\p_s\widetilde w+\widetilde I^\infty(\widetilde w)\p_t\widetilde w=0\quad\text{on }\R\times S^1.
\eea
Moreover, Proposition \ref{prop:gradbdv} implies
\bean
|\nabla\widetilde w(s,t)|\leq l\quad\text{for }(s,t)\in\R\times S^1.
\eea

From (\ref{eqn:T}) we know $T=\int_{\R\times S^1}v_\phi^*d\alpha<\infty$ and hence 
for every $R>0$ we have
\bean
\int_{[-R,R]\times S^1}w_k^*d\alpha_k
=\int_{[-R+s_k,R+s_k]\times S^1}v_\phi^*d\alpha\longrightarrow 0
\eea
as $k\to\infty$.
Thus we obtain
\bea\label{eqn:dlambda=0}
\int_{\R\times S^1}w^*d\alpha^\infty=0.\\
\eea

We know by the construction of $\widetilde v_\phi=(b_\phi,v_\phi)$ in (\ref{eqn:defofv}) 
that $v_\phi(s,t)$ converges to some point in $\widehat M$ as $s\to-\infty$.
Then for any $s_0\in\R$
\bean
\int_{\{s_0\}\times S^1}w_k^*\alpha_k
=\int_{(-\infty,s_0]\times S^1}w_k^*d\alpha_k
=\int_{(-\infty,s_0+s_k]\times S^1}v_\phi^*d\alpha
\eea
converges to $T=\int_{\R\times S^1}v^*d\alpha>0$ as $k\to\infty$.
Replace $k$ with $k'$ and passing to the limit $k'\to\infty$ we obtain
\bea\label{eqn:lambda>0}
\int_{\{s_0\}\times S^1}w^*\alpha^\infty=\int_{\R\times S^1}v_\phi^*d\alpha=T>0.
\eea

Now we consider an $\widetilde I^\infty$-holomorphic map
$\widetilde w_\rho:=(c_\rho,w_\rho):=(c,w)\circ \rho:\C\to\R\times\widehat M$ satisfying
\bea\label{eqn:Iinftyhol}
\pi^\infty_\alpha\p_s w_\rho+I^\infty(w_\rho)\p_t w_\rho&=0\\
w_\rho^*\alpha^\infty\circ i&=dc_\rho.
\eea
Here $\pi_\alpha^\infty:T\widehat M\to \ker\alpha^\infty$ be the projection along the Reeb direction
and $I^\infty=\widetilde I^\infty|_{\ker\alpha^\infty}$.
By the same argument as in Remark \ref{rmk:consthol} we deduce from (\ref{eqn:dlambda=0}) that 
$w^*d\alpha^\infty=0$ and hence
$w_\rho^*\alpha^\infty\in\Om^1(\C)$ is exact.
We introduce a primitive $q:\C\to\R$ of $w_\rho^*\alpha^\infty$ so that
\bean
\Phi:=c_\rho+iq:\C\to\C
\eea
is holomorphic.
Moreover, $w^*d\alpha^\infty=0$ implies that 
\bean
\pi_\alpha^\infty\circ dw_\rho(z):\C\to\ker\alpha^\infty(w_\rho(z))
\eea 
is the zero map for all $z\in\C$. 

Suppose that $\Phi$ is constant or $c_\rho$ is constant.
Then we have
\bean
w_\rho^*\alpha^\infty\circ i=dc_\rho=0
\eea
which means that $\p_s w_\rho, \p_t w_\rho$ have no Reeb direction component.
Since $\pi_\alpha^\infty\circ dw_\rho$ is also zero, $w_\rho$ must be a constant map.
But this contradicts
\bean
\int_{\{s_0\}\times S^1}w_\rho^*{\alpha^\infty}=T>0
\eea
and we conclude $\Phi$ is non-constant.
The gradient of $\Phi$ is bounded because of
\bean
\sup_{\C}|\nabla\Phi|=2\sup_{\C}|\nabla c_\rho|\leq 2\sup_{\C}|\nabla\widetilde w_\rho|=2\sup_{\C}|\nabla\widetilde w|<\infty.
\eea

As a consequence, $\Phi$ should be an affine non-constant holomorphic map, i.e.
\bean
\Phi(z)=lz+m
\eea
where $l,m\in\C$, $l=l_1+i\,l_2 \neq0$, $m=m_1+i\,m_2$. Then
\bean
c_\rho(z)=c_\rho(s,t)=l_1 s-l_2 t+m_1=l_1s+m_1
\eea
since $c_\rho$ is 1-periodic in $t$.
Let $X^\alpha_\infty$ be a Reeb vector field for the contact form $\alpha^\infty$ then we have
\bean
\p_s w_\rho&=\pi_\alpha^\infty\p_s w_\rho+{\alpha^\infty}(\p_s w_\rho)X^\alpha_\infty(w_\rho)\\
&={\alpha^\infty}(\p_s w_\rho)X^\alpha_\infty(w_\rho)\\
&=-\p_t c_\rho X^\alpha_\infty(w_\rho)\\
&=0
\eea
and
\bean
\p_t w_\rho&=\pi_\alpha^\infty\p_t w_\rho+{\alpha^\infty}(\p_t w_\rho)X^\alpha_\infty(w_\rho)\\
&={\alpha^\infty}(\p_t w_\rho)X^\alpha_\infty(w_\rho)\\
&=\p_s c_\rho X^\alpha_\infty(w_\rho)\\
&=l_1 X^\alpha_\infty(w_\rho).
\eea
Here we use $w_\rho^*\alpha^\infty\circ i=dc_\rho$ in (\ref{eqn:Iinftyhol}) for the 3rd equations.
Since $w_\rho$ is 1-periodic in the $t$-coordinate, 
we finally obtain a non-constant closed orbit $x(t):=w(s_0,l_1^{-1}t)$ of the Reeb vector field $X^\alpha_\infty$
which generate $\ker\widehat\om$.
Consequently a reparameterization of $\underline x(t)=p\circ x(t)$ gives us a non-constant contractible periodic orbit of the vector field $X^\om$ on $M$ which generates $\ker\om$.
\end{proof}

\begin{figure}[h]
\centering
\includegraphics[width=0.55\textwidth]{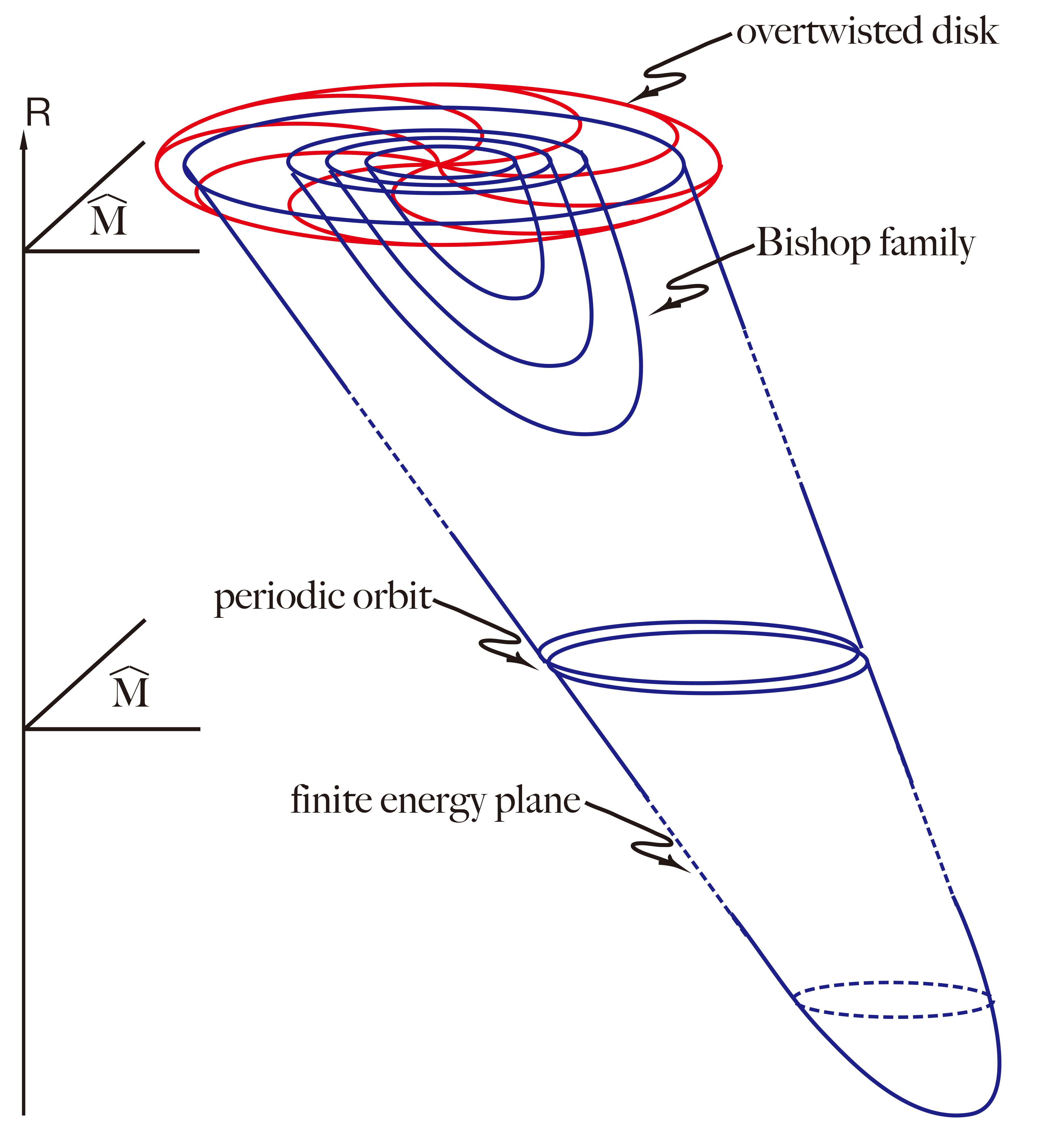}
\caption{schematic picture for the main theorem}
\end{figure}

\begin{proof}[Proof of Main Theorem]
By the overtwisted disk in $\widehat M$ and Theorem \ref{thm:gradexp}
we have a normalized maximal Bishop family $(\widetilde u_\tau)_{0\leq\tau<\tau_0}$ on $\R\times\widehat M$ with
\bean
\sup_{\tau\in[0,\tau_0)}\|\nabla\widetilde u_\tau\|_{C^0(\D)}=\infty.
\eea
From the rescaling argument with deck-transformation in the proof of Theorem \ref{thm:fep}
we construct an almost complex structure $\widetilde J^\infty$ as a limit and a finite energy $\widetilde J^\infty$-holomorphic plane \bean
\widetilde v=(b,v):\C\to\R\times\widehat M.
\eea
Finally Theorem \ref{thm:finmain} guarantees a contractible periodic orbit for the vector field $X^\om$ in Definition \ref{def:HS}
by considering the projection of the limit
\bean
x(t)=\lim_{R_k\to\infty}v(R_k e^{2\pi it}).
\eea
\end{proof}

\section{Examples of a virtual contact structure}

The well-known example of a virtual contact structure is an energy hypersurface of a twisted cotangent bundle above the Ma\~n\'e critical value, as mentioned in the introduction.
\begin{Ex}
On the cotangent bundle $\tau:T^*N\to N$ of a closed manifold $N$ 
we define autonomous Hamiltonian systems given by a convex Hamiltonian
\bean
H(q,p)=\frac{1}{2}|p|^2+U(q)
\eea
with a twisted symplectic form 
\bean\om_\sigma=dp\wedge dq+\tau^*\sigma.
\eea
Here $dp\wedge dq$ is the canonical symplectic form on $T^*N$, $|p|$ denote the dual norm of a Riemannian metric $g$ on $N$,
$U:\N\to\R$ is a smooth potential, and $\sigma$ is a closed 2-form on $N$.
When the pullback $\pi^*\sigma$ to the cover $\pi:\widehat N\to N$ is exact, the {\em Ma\~n\'e critical value} is defined as
\bean
c=c(g,\sigma,U,\pi):=\inf_{\theta}\sup_{q\in\widehat N}\widehat H(q,\theta_q),
\eea
where the infimum is taken over all $\theta\in\Om^1(\widehat N)$ satisfying $d\theta=\pi^*\sigma$ and $\widehat H$ is the lift of $H$ to $\widehat N$. Then $\Sigma_k:=H^{-1}(k), k>c$ admits a virtual contact structure 
\bean
\big(\pi:\pi^{-1}(\Sigma_k)\to\Sigma_k,\om_\sigma,\pi^*(p\wedge dq)+\theta_0\big)
\eea
for some primitive $\theta_0$ of $\pi^*\sigma$, see \cite[Lemma 5.1]{CFP} for the detail.
\end{Ex}

Other example is given by the mapping torus construction.
\begin{Ex}\label{ex:vcsHam}
Let $(L,\om_L)$ be a closed symplectic manifold and 
assume that $\om_L$ admits a bounded primitive on the universal cover $p:\widetilde L\to L$.
Now take a Hamiltonian diffeomorphism $\varphi$ on $(L,\om_L)$ and consider the induced mapping torus
\bean
L_\varphi:=\frac{L\times[0,1]}{(l,0)\sim(\varphi(l),1)}.
\eea
Then we can endow $L_\varphi$ with a virtual contact structure, see \cite[Remark 3.4]{Ba2} for the construction.
\end{Ex}

In the remaining section we construct a virtual contact structure with an overtwisted disk.
\begin{Ex}\label{ex:Lutz}
Let $\Sigma$ be a closed surface of genus $g\geq 2$
then $\Sigma$ can be represented by $\D/G$
with the universal covering map $p_\Sigma:\D\to\Sigma$.
Here $\D=\{x+iy\in\C\,|\,x^2+y^2<1\}$ is the Poincar\'{e} disk 
with the Poincar\'{e} metric
\bean
ds^2=\frac{dx^2+dy^2}{(1-x^2-y^2)^2}
\eea
and $G$ is a discrete group of isometry of $\D$.
Let $\om_\Sigma$ be a volume form on $\Sigma$
we then may assume
\bean
p_\Sigma^*\om_\Sigma
=\frac{2dx\wedge dy}{(1-x^2-y^2)^2}
=\frac{2rdr\wedge d\varphi}{(1-r^2)^2},
\eea
where $(r,\varphi)$ are the polar coordinates for $\C$.
Now consider $M=S^1\times\Sigma$ with the covering 
$p:S^1\times\D\to M$ and the projection
$\pi_\Sigma:M\to\Sigma$ which fit into the following diagram:
\bean
\xymatrix{
S^1\times\D
\ar[r]^{p} \ar[d]^{\pi}
&M
\ar[d]^{\pi_\Sigma}\\
\D \ar[r]^{p_\Sigma}
&\Sigma
}
\eea
The Hamiltonian structure $(M,\om_M:=\pi_\Sigma^*\om_\Sigma)$
admits a virtual contact structure by choosing 
a covering map $p:S^1\times\D\to M$ and
a primitive 
\bean
\alpha
=dt+\frac{xdy-ydx}{1-x^2-y^2}
=dt+\frac{r^2d\varphi}{1-r^2}
\eea
of $\om:=p^*\om_M$
where $t$ is the coordinate for $S^1$.\footnote{In Example \ref{ex:vcsHam}, let $(L,\om_L)$ be a closed surface $(\Sigma,\om_\Sigma)$ of genus $g\geq 2$ 
and take a Hamiltonian diffeomorphism $\varphi=\id$, then we have the above setting.}
Note here that $\om\in\Om^2(S^1\times\D)$ is $G$-invariant
but its primitive $\alpha\in\Om^1(S^1\times\D)$ is not.

In order to do a {\em  Lutz twist} on the virtual contact structure $(M,\om_M,\alpha)$, 
we start with a small constant
$\delta\in(0,\frac{1}{3}]$ satisfying
$g(B_\delta(0))\cap B_\delta(0)=\emptyset$ 
for all $g\in G$. 
Next we choose $\epsilon\in(0,\frac{\delta}{10}]$, $C=\max\{\|\alpha\|_\infty,1\}$
and then consider $C^1$-functions $f_1$, $f_2$ meeting the following conditions:
\begin{itemize}
\item $f_1(r)=-1$, $f_2(r)=\frac{-r^2}{1-r^2}$ for $r\in[0,\frac{\epsilon}{2})$;
\item $f_1(r)=1$, $f_2(r)=\frac{r^2}{1-r^2}$ for $r\in(\delta-\frac{\epsilon}{2},\delta]$;
\item $f_2(\frac{\delta}{2})=-4C$, $f_2'(\frac{\delta}{2})=0$.
\end{itemize}
The followings are not essential but we require additional conditions for simplicity:
\begin{itemize}
\item $f_1'$, $f_2'$ are piece-wise linear on $r\in[\frac{\epsilon}{2},\delta-\frac{\epsilon}{2}]$;
\item $f_1'$ is constant on $r\in[\epsilon,\delta-\epsilon]$;
\item $f_2'$ is constant on $r\in[\epsilon,\frac{\delta}{2}-\epsilon]\cup[\frac{\delta}{2}+\epsilon,\delta-\epsilon];$
\item $f_1'$ is differentiable except at $\{\frac{\epsilon}{2},\epsilon,\delta-\epsilon,\delta-\frac{\epsilon}{2}\}$;
\item $f_2'$ is differentiable except at $\{\frac{\epsilon}{2},\epsilon,\frac{\delta}{2}-\epsilon,\frac{\delta}{2},\frac{\delta}{2}+\epsilon,\delta-\epsilon,\delta-\frac{\epsilon}{2}\}$.
\end{itemize}
\begin{figure}[h]
\centering
\includegraphics[width=0.8\textwidth]{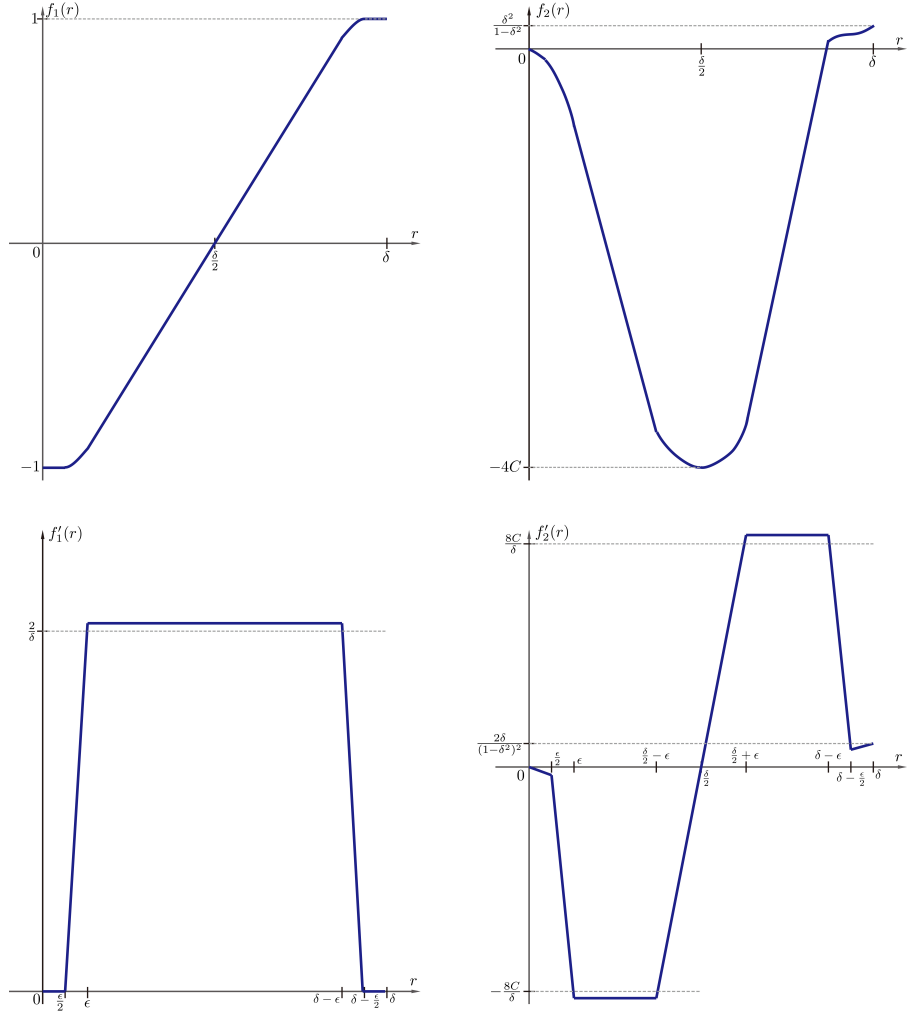}
\caption{graphs of $f_1(r)$, $f_2(r)$ and their differentials}
\end{figure}

Now we consider the 1-form 
\bea\label{eqn:alphaL}
\alpha^L_\delta:=f_1(r)dt+f_2(r)d\varphi
\eea 
on $\pi^{-1}(B_\delta(0))$
and $\eta_\delta\in\Om^1(S^1\times\D)$ 
which equals $\alpha^L_\delta-\alpha$ on $\pi^{-1}(B_\delta(0))$
and vanishes elsewhere.
From the locally supported 1-form $\eta_\delta$, 
we define a $G$-invariant 1-form $\eta\in\Om^1(S^1\times\D)$ by
\bean
\eta:=\sum_{g\in G}g^*\eta_\delta.
\eea
By the above construction there is a 1-form
$\underline\eta$ on $M$ such that $p^*\underline\eta=\eta$.
Since $\underline\eta$ is bounded, it immediately follows that
$\eta$ is also bounded with respect to 
the product metric on $S^1\times\D$.

\begin{Prop}
As in the above setting, a triple $(p:S^1\times \D\to M,\om^L:=\om_M+d\underline\eta,\alpha^L:=\alpha+\eta)$ is
a virtual contact structure equipped with an overtwisted disk.
\end{Prop}

\begin{proof}
We already know that $\alpha^L$ is bounded with respect to $G$-invariant metric.
Before checking the non-vanishing condition,
we first need to specify the non-vanishing vector field $R^L$ which generates $\ker d\alpha^L$.
For convenience, let $h:[\delta-\frac{\epsilon}{2},\delta]\to\R$ be a positive function satisfying
$h(\delta-\frac{\epsilon}{2})=f_2'(\delta-\frac{\epsilon}{2})$, $h(\delta)=1$.
Then we may choose our Reeb-like vector field $R^L$ as follows:
\bean
R^L=f_2'(\frac{\epsilon}{2})\p_t\quad&\text{ for }r\in[0,\frac{\epsilon}{2}];\\
R^L=f_2'(r)\p_t-f_1'(r)\p_\varphi\quad&\text{ for }r\in[\frac{\epsilon}{2},\delta-\frac{\epsilon}{2}];\\
R^L=h(r)\p_t\quad&\text{ for }r\in[\delta-\frac{\epsilon}{2},\delta];\\
R^L=\p_t\quad&\text{ for }x\in S^1\times \Big(\D\setminus\bigcup_{g\in G}g(B_\delta(0))\Big).
\eea
Recall that $R^L$ is $G$-invariant and hence $R^L$ on $S^1\times B_\delta(0)$
determines $R^L$ on $S^1\times g(B_\delta(0))$ for any $g\in G$.

Let us begin with the case $x\in S^1\times \big(\D\setminus\bigcup_{g\in G}g(B_\delta(0))\big)$.
In this region there is no change under the Lutz twist and $g\in G$ acts trivially on $\p_t$, thus we have
\bean
\alpha^L(x)[R^L(x)]=1.
\eea
Next we consider the case $r\in[0,\frac{\epsilon}{2})$ where 
$\alpha^L=-dt-\frac{r^2}{1-r^2}d\varphi$ and $R^L=f_2'(\frac{\epsilon}{2})\p_t$.
For $x\in S^1\times\big(\bigcup_{g\in G}g(B_{\frac{\epsilon}{2}}(0))\big)$ we obtain
\bean
\alpha^L(x)[R^L(x)]=-f_2'(\frac{\epsilon}{2})>0.
\eea
Now we move to the case $r\in[\delta-\frac{\epsilon}{2},\delta]$ 
where $\alpha^L=dt+\frac{r^2}{1-r^2}d\varphi$ and $R^L=h(r)\p_t$.
Let $A=\{(r,\varphi)\in\D\,|\,\delta-\frac{\epsilon}{2}\leq r\leq\delta\}$ and 
$x\in S^1\times\big(\bigcup_{g\in G}g(A)\big)$ then
\bean
\alpha^L(x)[R^L(x)]\geq\min_{r\in[\delta-\frac{\epsilon}{2},\delta]}h(r)>0.
\eea

It remains to verify the case $r\in[\frac{\epsilon}{2},\delta-\frac{\epsilon}{2}]$ and
we need the following preparation.
Let $g\in G$ be a M\"obius transformation which acts on $\D$ and denote $g(0)=z_0$.
Then it is necessary to know the lower bound of $\alpha^L(R^L)$
on $\pi^{-1}\circ g(B_\delta(0))$.
The following simple observation
\bean
\alpha^L(R^L)|_{\pi^{-1}(z_0)}&=\alpha^L(g_*R^L)|_{g(\pi^{-1}(0))}
=g^*\alpha^L(R^L)|_{\pi^{-1}(0)}
\eea
informs us that it suffices to investigate $g^*\alpha^L(R^L)$
on $\pi^{-1}(B_\delta(0))$.
Now we compare $g^*\alpha^L$ with $\alpha^L$ as follows:
\bea\label{eqn:gpulla}
g^*\alpha^L(R^L)&=g^*(\alpha+\eta)(R^L)\\
&=(g^*\alpha+\eta)(R^L)\\
&=(g^*\alpha-\alpha+\alpha^L)(R^L).
\eea
Thus it is important to estimate $(g^*\alpha-\alpha)(R^L)$.
Since $G$ acts trivially on the $t$-coordinate,
$g^*\alpha-\alpha$ has no $dt$-part and by the construction of $C$ we know $\|\alpha\|\leq C$. 
For $r\in[\frac{\epsilon}{2},\delta-\frac{\epsilon}{2}]$ we then have
\bea\label{eqn:est1}
|(g^*\alpha-\alpha)(R^L)|
&\leq \|g^*\alpha-\alpha\|\cdot\|f_2'(r)\p_t-f_1'(r)\p_\varphi\|\\
&\leq 2Cf_1'(r)\|\p_\varphi\|\\
&\leq 2Cf_1'(r).
\eea
Here the last inequality comes from the following estimate:
\bean
\|\p_\varphi\|&=\|-r\sin\varphi\p_x+r\cos\varphi\p_y\|\\
&\leq\|r\p_x\|+\|r\p_y\|\\
&\leq\frac{2r}{1-r^2}\Big|_{r=\delta-\frac{\epsilon}{2}}\\
&\leq1.
\eea
The point of the estimate (\ref{eqn:est1}) is that the last term does not depend on $f_2'(r)$ and $g\in G$.
By combining (\ref{eqn:alphaL}), (\ref{eqn:gpulla}) and (\ref{eqn:est1}),
we estimate 
\bean
g^*\alpha^L(R^L)
&\geq\alpha^L(R^L)-|(g^*\alpha-\alpha)(R^L)|\\
&\geq \underbrace{f_1(r)f_2'(r)-f_2(r)f_1'(r)-2Cf_1'(r)}_{=:L(r)},\\
\eea
for $r\in[\frac{\epsilon}{2},\delta-\frac{\epsilon}{2}]$.\\[-2mm]

\noindent {\bf Case i} : $r\in[\frac{\epsilon}{2},\epsilon]$.

Recall that $C\geq1$, $\delta\in(0,\frac{1}{3}]$, $\epsilon\in(0,\frac{\delta}{10}]$.
By the choice of $\delta,\epsilon$
we estimate
\bean
f_1(r)&=\frac{4}{2\epsilon\delta-3\epsilon^2}(r-\frac{\epsilon}{2})^2-1\leq-\frac{7}{8};\\
f_1'(r)&=\frac{8}{2\epsilon\delta-3\epsilon^2}(r-\frac{\epsilon}{2})\leq\frac{5}{2\epsilon\delta}(r-\frac{\epsilon}{2});\\
f_2'(r)&\leq\frac{-16C+2\epsilon\delta}{\epsilon\delta}(r-\frac{\epsilon}{2})-\epsilon\\
\eea
and
\bean
L(r)&>f_1(r)f_2'(r)-2Cf_1'(r)\\
&\geq-\frac{7}{8}\left(\frac{-16C+2\epsilon\delta}{\epsilon\delta}(r-\frac{\epsilon}{2})-\epsilon\right)-2C\frac{5}{2\epsilon\delta}(r-\frac{\epsilon}{2})\\
&=\frac{36C-7\epsilon\delta}{4\epsilon\delta}(r-\frac{\epsilon}{2})+\frac{7}{8}\epsilon\\
&\geq\frac{7}{8}\epsilon.
\eea\\[-2mm]

\noindent {\bf Case ii} : $r\in[\epsilon,\frac{\delta}{2}-\epsilon]$.

We have the following estimate
\bean
f_1(r)&=\frac{4}{2\delta-3\epsilon}(r-\frac{\delta}{2});\\
f_1'(r)&=\frac{4}{2\delta-3\epsilon};\\
f_2(r)&\leq-\frac{8C}{\delta}(r-\epsilon);\\
f_2'(r)&\leq-\frac{8C}{\delta}
\eea
and
\bean
L(r)&\geq\frac{4}{2\delta-3\epsilon}(r-\frac{\delta}{2})(-\frac{8C}{\delta})+\frac{8C}{\delta}(r-\epsilon)\frac{4}{2\delta-3\epsilon}
-2C\frac{4}{2\delta-3\epsilon}\\
&=\frac{8C}{2\delta-3\epsilon}(1-\frac{4\epsilon}{\delta})\\
&\geq\frac{8C}{2\delta-3\epsilon}\cdot\frac{3}{5}\\
&\geq\frac{12C}{5\delta}.
\eea\\[-2mm]

\noindent{\bf Case iii} : $r\in[\frac{\delta}{2}-\epsilon,\frac{\delta}{2}+\epsilon]$.

We obtain
\bean
f_1(r)&=\frac{4}{2\delta-3\epsilon}(r-\frac{\delta}{2});\\
f_1'(r)&=\frac{4}{2\delta-3\epsilon};\\
f_2(r)&\leq-3C;\\
|f_2'(r)|&\leq|\frac{8C}{\epsilon\delta}(r-\frac{\delta}{2})|
\eea
and
\bean
L(r)&\geq\frac{4}{2\delta-3\epsilon}(r-\frac{\delta}{2})\cdot\frac{8C}{\epsilon\delta}(r-\frac{\delta}{2})+3C\frac{4}{2\delta-3\epsilon}-2C\frac{4}{2\delta-3\epsilon}\\
&=\frac{4C}{2\delta-3\epsilon}\left(\frac{8}{\epsilon\delta}(r-\frac{\delta}{2})^2+1\right)\\
&\geq\frac{4C}{2\delta-3\epsilon}\\
&\geq\frac{2C}{\delta}.
\eea
The cases $r\in[\frac{\delta}{2}+\epsilon,\delta-\epsilon]$, $r\in[\delta-\epsilon,\delta-\frac{\epsilon}{2}]$ are
similar to {\bf Case ii}, {\bf Case i} respectively.

Therefore we finally have $\inf_{x\in S^1\times\D}\alpha^L(R^L)>0$ which conclude that the triple
$(p:S^1\times\D\to M,\om^L,\alpha^L)$ is a virtual contact structure.
Moreover, $\pi^{-1}(B_\delta(0))$ contains an overtwisted disk.
This proves the lemma.
\end{proof}
\end{Ex}
In the above construction the 1-form $\alpha^L$ and the vector field $R^L$ are not smooth. 
In fact $f_1'$, $f_2'$ are not smooth on finite points and $R^L$ is also not smooth on the corresponding region.
Now we consider a sequence of $C^\infty$ functions $(h_{1,n}, h_{2,n})_{n\in\N}$ which converges to $(f_1, f_2)$ in $C^1$.
Then by the same construction as above we have a sequence of smooth 1-forms $(\lambda_n^L)_{n\in\N}$ 
and a sequence of smooth vector fields $(X_n^L)_{n\in\N}$ 
satisfying the property that $\lambda_n^L(X_n^L)$ converges uniformly to $\alpha^L(R^L)$.
For sufficiently large $n\in\N$ we choose a smooth 1-form $\lambda^L$ such that $\inf_{x\in S^1\times\D}\lambda^L(X^L)>0$.
Again by the construction, each derivative of $\lambda^L_n$ is determined by $h_{1,n}, h_{2,n}$.
We may require that all derivatives of $h_{1,n}, h_{2,n}$ 
are uniformly bounded with respect to the Poincar\'e metric, without loss of generality.
These conclude that $(p:S^1\times\D\to M,\om^L,\lambda^L)$ is a smooth virtual contact structure.

\end{document}